\documentclass[preprint,12pt,final]{elsarticle}
\usepackage{epsfig,amssymb,amsmath,version,amsthm}
\usepackage{amssymb,version,graphicx,fancybox,mathrsfs}
\usepackage{url,hyperref}
\usepackage[notcite,notref]{showkeys}

\usepackage{subfigure}
\usepackage{color,xcolor}
\usepackage{cases}
\usepackage{mathtools}
\usepackage{booktabs}
\usepackage{multirow}

\journal{}

\newtheorem{cor}{\bf Corollary}[section]

\theoremstyle{remark}

\numberwithin{equation}{section}
\newtheorem{theorem}{\bf Theorem}[section]
\newtheorem{definition}{\bf Definition}[section]
\newtheorem{lemma}{\bf Lemma}[section]

\numberwithin{table}{section}

\def \ri {{\rm i}}

\newcommand{\bs}[1]{\boldsymbol{#1}}

\allowdisplaybreaks
\graphicspath{{../figure/}}
\begin{document}

	\begin{frontmatter}

		\title{High-order implicit Galerkin-Legendre spectral method for the two-dimensional Schr\"{o}dinger equation\tnoteref{label1}}
		
		\tnotetext[label1]{``The work of  authors was partially supported by the National Natural Science Foundation of China (11271100, 11301113, 51476047, 11601103) and the Fundamental Research Funds for the Central Universities (Grant No. HIT. IBRSEM. A. 201412), Harbin Science and Technology Innovative Talents Project of Special Fund (2013RFXYJ044)."}

	\author {Wenjie Liu$^a$, Boying Wu$^a$}

	\address[fn1]{Department of Mathematics, Harbin  Institute of Technology, 150001, China}

		\begin{abstract}
		In this paper, we propose Galerkin-Legendre spectral method
		with implicit Runge-Kutta method
		for solving the unsteady two-dimensional Schr\"{o}dinger equation with nonhomogeneous Dirichlet boundary conditions and initial condition.
		We apply a Galerkin-Legendre spectral method for discretizing spatial derivatives, and then
		employ the implicit Runge-Kutta method for the time
		integration of the resulting linear first-order system
		of ordinary differential equations in complex domain.
		We derive the spectral rate of convergence for the proposed method in the $L^2$-norm for the semidiscrete formulation.
		Numerical experiments show our formulation have high-order accurate.
		\end{abstract}
		
		
		
	\end{frontmatter}
	
	\section{Introduction}
	In this paper, we introduce  Galerkin-Legendre spectral method
	for the two-dimensional Schr\"{o}dinger equation
	\begin{subequations}\label{Introduction-1}
		\begin{align}
		\label{Introduction-1new}
		- \ri\frac{ \partial
			u}{\partial t}= \Delta u+\psi(x,y) u,\ (x,y,t) \in (c,d)\times (c,d)\times (t_0,T),
		\end{align}
		with the initial condition
		\begin{align}
		\label{Introduction-2}
		u(x,y,t_0) = \varphi(x,y) ,\ (x,y) \in \Omega = (c,d)\times (c,d),
		\end{align}
		and the Dirichlet boundary conditions
		\begin{equation}\begin{split}
		\label{Introduction-3}
		&u(x,c,t) = g_1(x,t),\ u(x,d,t) = g_2(x,t),\; (x,t)\in(c,d)\times (t_0,T) ,\\
		&u(c,y,t) = g_3(y,t),\ u(d,y,t) = g_4(y,t),\; (y,t)\in(c,d)\times (t_0,T) ,\\
		\end{split}\end{equation}
	\end{subequations}
	where $u$ is a complex-valued function, $\psi$ is a smooth known real function, $ \varphi$ and $\{g_j\}_{j=1}^{4}$ are smooth known complex functions,  and $\ri =\sqrt{-1}$.
	
	The Schr\"{o}dinger equation is
	a famous equation used widely in many fields of physics \cite{Mohebbi2009paper,Abdur2011paper,Dehghan2013paper},
	such as quantum mechanics, quantum
	dynamics calculations, optics, underwater acoustics, plasma physics and electromagnetic wave propagation.
	
	Over the past few years, several numerical schemes have been
	developed for solving problem \eqref{Introduction-1}.
	For instance,
	In \cite{Subasi2002paper}
	Subasi used a standard finite difference method in space for solving the problem \eqref{Introduction-1}.
	A Crank-Nicolson method discretization for the time of the problem \eqref{Introduction-1} was considered in in \cite{Antoine2004paper}.
	An implicit semi-discrete higher order compact (HOC) scheme was considered for computing
	the solution of the problem \eqref{Introduction-1} in \cite{Kalita2006paper}.
	Dehghana and Shokriin \cite{Dehghan2007paper} proposed a numerical scheme
	to solve the problem \eqref{Introduction-1} by collocation and radial basis functions in space.
	A meshless local boundary integral equation (LBIE) method to solve the problem \eqref{Introduction-1} was given
	in \cite{Dehghan2008paper}.
	In \cite{Mohebbi2009paper} Mohebbi and Dehghan presented a high order method for the problem \eqref{Introduction-1} by
	the compact finite difference in space
	and boundary value method in time.
	In \cite{Tian2010paper} Tian and Yu
	proposed a HOC-ADI method to solve the problem \eqref{Introduction-1}, which has fourth-order accuracy in
	space and second-order accuracy in time.
	Gao and Xie \cite{Gao2011paper} proposed a numerical scheme
	to solve the problem \eqref{Introduction-1}
	by the ADI compact finite difference
	scheme.
	For the spatial discretisation of the problem \eqref{Introduction-1} by
	the Chebyshev spectral collocation
	method
	considered in \cite{Abdur2011paper}.
	In \cite{Dehghan2013paper}
	Dehghan and Emami-Naeini presented Sinc-collocation and Sinc-Galerkin methods to solve the
	problem \eqref{Introduction-1}.

	In this paper, we propose a numerical scheme for solving \eqref{Introduction-1}.
	We apply Galerkin-Legendre spectral method \cite{Shen2006Book,Shen2011Book,shj1}
	for discretizing spatial derivatives, then using
	the implicit Runge-Kutta method for time derivatives,
	which is high-order accurate both in space and time. 

	The contents of the article is as follows.
	In Section 2, we introduce the implicit Runge-Kutta method for the system of ordinary differential equations in the complex domain.
	In Section 3, we describe the Galerkin-Legendre spectral
	method for the two-dimensional Schr\"{o}dinger equation in space.
	In Section 4, we derive a priori error estimates in the $L^2$-norm
	for the semidiscrete formulation. The analysis relies on an idea suggested by
	Lions et al. \cite{Lions1972Book} and Thom\'{e}e \cite{Thomee2006Book}.
	In Section 5, we report the numerical experiments of solving
	two-dimensional Schr\"{o}dinger equation with the new method developed in this paper,
	and compare the numerical results with analytical solutions and with other method in the literature \cite{Mohebbi2009paper}.
	We end this article with some concluding remarks
	in Section 6.
	\section{Implicit Runge-Kutta method}
	
	In this section we modifed the implicit Runge-Kutta (IRK) method to solve first-order 
	linear complex ordinary differential equations.
	We first give the definition of Kronecker product of matrices.
	\begin{definition}[see, e.g., \cite{Steeb2006Book}]
		Let $\bs A = (a_{ij})_{m\times n}$ and $\bs B$ be arbitrary matrices, then the matrix
		$$\bs A \otimes\bs B =\left(
		\begin{array}{cccc}
		a_{11}\bs B & a_{12}\bs B & \ldots & a_{1n}\bs B\\
		a_{21}\bs B & a_{22}\bs B & \ldots & a_{2n}\bs B \\
		\vdots & \vdots & \ddots &\vdots \\
		a_{m1}\bs B & a_{m2}\bs B & \ldots& a_{mn}\bs B \\
		\end{array}
		\right),
		$$
		is called the Kronecker product of $\bs A$ and $\bs B$.
	\end{definition}
	\begin{definition}[see, e.g., \cite{Steeb2006Book}]
		Let $\bs A=(a_{ij})_{m\times n}$ be any given matrix, then ${\rm vec}(\bs A)$ is defined to be a column
		vector of size $m \times n$ made of the row of
		$$
		{\rm vec}(\bs A) = (a_{11},a_{12},\ldots,a_{1n},a_{21},a_{22},\ldots,a_{m1},\ldots,a_{mn})^T.
		$$
	\end{definition}
	
	\begin{lemma}[see, e.g., \cite{Steeb2006Book}]
		\label{lemmasp1}
		Let $\bs A =(a_{ij})_{m\times n}$, $\bs B =(b_{ij})_{n\times p}$, and $\bs C =(c_{ij})_{p\times q}$ be three given matrices. Then
		$$
		{\rm vec}(\bs A\bs B\bs C) =(\bs A\otimes \bs C^T){\rm vec}(\bs B).
		$$
	\end{lemma}
	
	\begin{lemma}[see, e.g., \cite{Steeb2006Book}]
		\label{lemmasp2} Let $\bs A =(a_{ij})_{m\times 1}$, $\bs B =(b_{ij})_{n\times
			1}$. Then
		$$
		{\rm vec}(\bs A\bs B^T) =\bs A\otimes \bs B.
		$$
	\end{lemma}

	We consider the following first-order initial value problem is given by
	\begin{equation}\label{IRKM2-1}
	\begin{cases}
	y'=f(t,y(t)),\;\; t_0< t\leq T,
	\\y(t_0)=y_0.
	\end{cases}
	\end{equation}
	A $s$-stage formula of implicit Runge-Kutta method for approximating  \eqref{IRKM2-1} can be
	written as
	\begin{equation}\label{IRKM2-2}
	\begin{cases}
	y_{n + 1} = {y_n} + h\sum\limits_{l = 1}^s {{b_l}} f\left( {{t_n} + {c_l}h,{k_l}} \right),\\
	{k_l} = {y_n} + h\sum\limits_{j = 1}^s {{a_{lj}}} f\left( {{t_n} + {c_j}h,{k_j}} \right),\quad l= 1,2, \ldots ,s,
	\end{cases}
	\end{equation}
	where $h=(T-t_0)/(N_t-1)$ is the integration step, $\{k_l\}_{l=1}^s$ are the internal stages and $t_n = t_0+hn$.
	If $a_{lj}=0\, (l>j)$ the method is called diagonally implicit Runge-Kutta (DIRK) method.
	Then the  \eqref{IRKM2-2} can be written the following form
	\begin{equation}
	\label{IRKM2-3}
	\begin{cases}
	\bs K =  {\bs 1}_s y_n+h \bs A \bs F( \bs K),\\
	y_{n + 1} = {y_n} + h \bs B^T \bs F(\bs K),
	\end{cases}
	\end{equation}
	where $A=(a_{lj})_{l,j=1,\ldots,s},$ $\bs B=[b_1,b_2,\cdots,b_s]^T$, $\bs K=[k_1,k_2,\cdots,k_s]^T$, ${\bs 1}_s = [\underbrace{1,1,\cdots,1}_s]^T$,
	and
	\begin{equation*}
	\bs F(\bs K)=[f( t_n + c_1h,k_1),f( t_n + c_2h,k_2),\cdots,f( t_n + c_sh,k_s)]^T.
	\end{equation*}
	
	If we consider the system of linear ordinary differential equations
	\begin{align}
	\label{IRKM2-4}
	\begin{cases}
	-\ri \bs M_1\bs y'+\bs M_2\bs y-\bs  f(t)=0, \ \ t_0< t \leq T,
	\\ \bs y(t_0)=\bs y_0,
	\end{cases}
	\end{align}
	where $\bs M_1,\bs M_2 \in \mathbb{R}_{m\times m}$, and
	\begin{align*}\bs y(t)=[y_1(t),y_2(t),\cdots,y_m(t)]^T,
	\;\; \bs f(t)=[f_1(t),f_2(t),\cdots,f_m(t)]^T.
	\end{align*}
	Using s-stage formula of Runge-Kutta method \eqref{IRKM2-3} for approximating  \eqref{IRKM2-4} can be written as
	\begin{equation}
	\label{IRKM2-5}
	\begin{cases}
	-\ri  \bs K \bs M_1^T = -\ri  {\bs 1}_s \bs y_n^T \bs M_1^T+h \bs A(-\bs K \bs M_2^T+\bs F),\\
	-\ri \bs y_{n + 1}^T\bs M_1^T = -\ri {\bs y_n^T}\bs M_1^T + h\bs B^T (-\bs K\bs M_2^T+\bs F),
	\end{cases}
	\end{equation}
	where
	\begin{equation*}\begin{split}
	&\bs K=[\bs k_1^T;\bs k_2^T;\cdots;\bs k_s^T],\;\;
	\bs F=[\bs f( t_n + c_1h)^T; \bs f( t_n + c_2h)^T;\cdots;\bs f( t_n + c_sh)^T].
	\end{split}\end{equation*}
	By using Lemma \ref{lemmasp1} and \ref{lemmasp2}, we have
	\begin{equation}
	\label{IRKM2-8}
	(-\ri \bs I_{s}\otimes \bs M_1 + h\bs  A\otimes\bs  M_2) {\rm vec}(\bs K) =-\ri {\bs 1}_s\otimes  (\bs M_1{\bs y}_n)+ h  \bs A\otimes\bs  I_m {\rm vec}(\bs F),
	\end{equation}
	where $\bs I_s$ is the $s\times s$  unitary matrix.
	Thus we can obtain ${\rm vec}(\bs K)$ from \eqref{IRKM2-8} with
	GMRES iteration (see, e.g., \cite[PP. 55--58]{Shen2006Book}).
	We write the Runge-Kutta scheme in a tabular format known as the Butchers table,
	in Table \ref{ThreeIRK}, we choose a 3-stage IRK method (cf. \cite{Butcher}).
	\begin{table}[!htp]
		\centering
		\caption{Left: Butchers table. Right: 3-stage IRK method} 	  	
		\begin{tabular}{c|cccc}
			$c_1$ & $a_{11}$ & $\cdots$ & $a_{1s}$\\
			$\vdots$ & $\vdots$ & $\ddots$ & $\vdots$ \\
			$c_s$ & $a_{11}$ & $\cdots$ & $a_{ss}$\\
			\hline
			$\,$ & $b_{1}$ & $\cdots$ & $b_{s}$\\
		\end{tabular}\qquad
		\begin{tabular}{c|cccc}
			$\frac{1}{2}-\frac{\sqrt{15}}{10}$ & $\frac{5}{36}$ & $\frac{2}{9}-\frac{\sqrt{15}}{15}$ & $\frac{5}{36}-\frac{\sqrt{15}}{30}$\\
			$\frac{1}{2}$ & $\frac{5}{36}+\frac{\sqrt{15}}{24}$ & $\frac{2}{9}$ & $\frac{5}{36}-\frac{\sqrt{15}}{24}$ \\
			$\frac{1}{2}+\frac{\sqrt{15}}{10}$ & $\frac{5}{36}+\frac{\sqrt{15}}{30}$ & $\frac{2}{9}+\frac{\sqrt{15}}{15}$ & $\frac{5}{36}$\\
			\hline
			$\,$ & $\frac{5}{18}$ & $\frac{4}{9}$ & $\frac{5}{18}$\\
		\end{tabular}
		\label{ThreeIRK}
	\end{table}
	
	\section{Discretize two-dimensional Schr\"{o}dinger equation in space by Galerkin-Legendre spectral method}
	
	In this section,
	we will present the Galerkin-Legendre spectral method to solve
	the unsteady two-dimensional Schr\"{o}dinger equation for the space with the nonhomogeneous Dirichlet
	boundary conditions and initial condition.
	Firstly, we make the variable transformations
	\begin{align*}
	&x=c+\frac{\tilde{x}+1}{2}(d-c),\ y=c+\frac{\tilde{y}+1}{2}(d-c),\ \gamma =\Big (\frac{2}{d-c}\Big)^2,\\
	&\tilde{u}(\tilde{x},\tilde{y},t)=u\big(c+\frac{\tilde{x}+1}{2}(d-c),c+\frac{\tilde{y}+1}{2}(d-c),t\big),\\
	&\tilde{\psi}(\tilde{x},\tilde{y}) = \psi\big(c+\frac{\tilde{x}+1}{2}(d-c),c+\frac{\tilde{y}+1}{2}(d-c)\big),\\
	&\tilde{\varphi}(\tilde{x},\tilde{y}) = \varphi\big(c+\frac{\tilde{x}+1}{2}(d-c),c+\frac{\tilde{y}+1}{2}(d-c)\big),\\
	&\tilde{g}_1(\tilde{y},t) =g_1\big(c+\frac{\tilde{y}+1}{2}(d-c),t\big),\
	\tilde{g}_2(\tilde{y},t) =g_2\big(c+\frac{\tilde{y}+1}{2}(d-c),t\big),\\
	&\tilde{g}_3(\tilde{x},t) =g_3\big(c+\frac{\tilde{x}+1}{2}(d-c),t\big),\
	\tilde{g}_4(\tilde{x},t) =g_4\big(c+\frac{\tilde{x}+1}{2}(d-c),t\big).
	\end{align*}
	Then $\Omega$ is changed to the square $\widetilde{\Omega} = (-1,1)\times (-1,1)$,
	and the \eqref{Introduction-1} can be rewritten as
	\begin{subequations}\label{1-2}
		\begin{align}
		-\ri \frac{\partial \tilde{u}}{\partial t}=\gamma\Delta \tilde{u}
		+\tilde{\psi}(\tilde{x},\tilde{y})\tilde{u} ,\;\;(\tilde{x},\tilde{y},t)\in \widetilde{\Omega}\times (t_0,T),
		\end{align}
		with the initial condition
		\begin{equation}
		\tilde{u}(\tilde{x},\tilde{y},t_0)=\tilde{\varphi}(\tilde{x},\tilde{y}),\ {\rm on}
		\;\;	\widetilde{\Omega},
		\end{equation}
		and the Dirichlet boundary conditions
		\begin{equation}\begin{split}
		&\tilde{u}(-1,y,t)=\tilde{g}_1(y,t),\ \tilde{u}(1,\tilde{y},t)=\tilde{g}_2(\tilde{y},t),\ {\rm on}\ J=(t_0,T),\\
		&\tilde{u}(\tilde{x},-1,t)=\tilde{g}_3(\tilde{x},t),\ \tilde{u}(\tilde{x},1,t)=\tilde{g}_4(\tilde{x},t),\ {\rm on}\ J=(t_0,T).
		\end{split}\end{equation}
	\end{subequations}

	Now we recall the boundary conditions homogeneous process (see, e.g., \cite{shj1}).
	Setting
	\begin{align*}
	&\tilde{u}_1(\tilde{x},\tilde{y},t)= \frac{\tilde{g}_4(\tilde{x},t)-\tilde{g}_3(\tilde{x},t)}{2}\tilde{y}+\frac{\tilde{g}_4(\tilde{x},t)+\tilde{g}_3(\tilde{x},t)}{2},\\
	&\hat{g}_1(\tilde{y},t)=\tilde{g}_1(\tilde{y},t)-\tilde{u}_1(-1,\tilde{y},t),\ \hat{g}_2(\tilde{y},t)=\tilde{g}_2(\tilde{y},t)-\tilde{u}_1(1,\tilde{y},t),\\
	&\tilde{u}_2(\tilde{x},\tilde{y},t)=  \frac{\hat{g}_2(\tilde{y},t)-\hat{g}_1(\tilde{y},t)}{2}\tilde{x}+\frac{\hat{g}_2(\tilde{y},t)+\hat{g}_1(\tilde{y},t)}{2},\\
	&\hat{u} = \tilde{u}- \tilde{u}_1-\tilde{u}_2,\\
	&\tilde{f}(\tilde{x},\tilde{y},t)= \ri \Big(\frac{\partial \tilde{u}_1}{\partial t}+\frac{\partial \tilde{u}_2}{\partial t}\Big)+\gamma
	(\Delta \tilde{u}_1+\Delta \tilde{u}_2)+\tilde{\psi}(\tilde{u}_1+ \tilde{u}_2),\\
	&\hat{\varphi}(\tilde{x},\tilde{y}) = \tilde{\varphi}(\tilde{x},\tilde{y})-\tilde{u}_1(\tilde{x},\tilde{y},t_0)-\tilde{u}_2(\tilde{x},\tilde{y},t_0),
	\end{align*}
	then the  \eqref{1-2} can be rewritten as
	\begin{subequations}\label{1-3}
		\begin{align}\label{1-3p}
		-\ri \frac{\partial \hat{u}}{\partial t}=\gamma\Delta \hat{u}+\tilde{\psi}(\tilde{x},\tilde{y})\hat{u}
		+\tilde{f}(\tilde{x},\tilde{y},t),\ (\tilde{x},\tilde{y},t)\in (-1,1)\times (-1,1)\times (t_0,T),
		\end{align}
		with the initial condition and homogeneous boundary value conditions
		\begin{equation}\begin{split}
		\hat{u}(\tilde{x},\tilde{y},t_0)=&\hat{\varphi}(\tilde{x},\tilde{y}),\, {\rm on}\;
		\widetilde{\Omega},\\
		\hat{u}(\tilde{x},-1,t)=\hat{u}(\tilde{x},1,t)=&\hat{u}(-1,\tilde{y},t)=\hat{u}(1,\tilde{y},t)=0,
		\, {\rm on}\; J=(t_0,T).
		\end{split}\end{equation}
	\end{subequations}
	We shall now discretize the equation \eqref{1-3}
	by using the Galerkin-Legendre spectral method in space. 
	Let us denote $L_n(\tilde{x})$ the $n$th degree Legendre polynomial (see, e.g., \cite[PP. 18 and 19]{Shen2006Book}) and
	\begin{equation*}
	P_N=\textup{span}\{L_0(\tilde{x}),L_1(\tilde{x}),\cdots,L_N(\tilde{x})\},\ V_N=\{v\in P_N:v(\pm
	1)=0\}.
	\end{equation*}
	Then the semi-discrete Legendre-Galerkin method for \eqref{1-3} is:
	find $\hat{u}_N:\bar{J}\rightarrow V_N^2 $ such that
	\begin{align}
	\label{GLP-1} \left\{\begin{array}{ll} -\ri (\partial
	{_t}\hat{u}_N, v)+\gamma( \nabla
	\hat{u}_N,\nabla v)-(\tilde{\psi}(\tilde{x},\tilde{y})\hat{u}_N,v)-(\hat{f}(\tilde{x},\tilde{y},t),v)=0,\quad \forall\, v \in V_N^2,
	\\
	(\hat{u}_N(\tilde{x},\tilde{y},t_0)-\hat{u}_0,v)=0,\quad\forall \,v \in V_N^2,
	\end{array}\right.
	\end{align}
	where $(u,v)=\int_{\widetilde{\Omega}}u\bar{v}d\tilde{x}d\tilde{y}$ is the scalar
	product in $ L^2(\widetilde{\Omega})$, the $\bar{v}$ is complex conjugate of $v$.
	
	The following lemma is the key to implement our algorithms.
	\begin{lemma}\textup{(\cite[Lemma 2.1]{shj1})} 
		\label{lemma1}  Let us denote
		\begin{displaymath}
		\phi_k(\tilde{x})=c_k(L_k(\tilde{x})-L_{k+2}(\tilde{x})),\;\;c_k=\frac{1}{\sqrt{4k+6}},
		\end{displaymath}
		\begin{displaymath}
		\tilde{a}_{jk}=\int^1_{-1}\phi'_k(\tilde{x})\phi'_j(\tilde{x})d\tilde{x},\;\; \tilde{b}_{jk}=\int^1_{-1}\phi_k(\tilde{x})\phi_j(\tilde{x})d\tilde{x} ,
		\end{displaymath}
		then
		\begin{equation*}
		\tilde{a}_{jk}=\left\{\begin{array}{ll} 1, & k=j,
		\\
		0, &k\neq j,
		\end{array} \right.
		\quad\tilde{ b}_{jk}=\tilde{b}_{kj}=\left\{\begin{array}{ll}
		c_kc_j(\frac{2}{2j+1}+\frac{2}{2j+5}), & k=j,
		\\
		-c_kc_j\frac{2}{2k+1}, & k=j+2,
		\\
		0, & otherwise.
		\end{array} \right.
		\end{equation*}
	\end{lemma}
	Obviously
	$$
	V_N=\textup{span}\{\phi_0(\tilde{x}),\phi_1(\tilde{x}),\cdots,\phi_{N-2}(\tilde{x})\}.
	$$
	Let us setting
	\begin{align}
	\label{GLP-2}\hat{u}_N(\tilde{x},\tilde{y},t)=\sum_{k,j=0}^{N-2}
	\alpha_{kj}(t)\phi_k(\tilde{x})\phi_j(\tilde{y}),
	\end{align}
	then inserting \eqref{GLP-2} into \eqref{GLP-1} and taking
	\begin{equation*}
	v = \phi_l(\tilde{x})\phi_m(\tilde{y}),
	\end{equation*}
	yields
	\begin{equation}\label{GLP-3}
	\begin{aligned}
	-\ri \big(\partial {_t}\hat{u}_N,\phi_l(\tilde{x})\phi_m(\tilde{y})\big)&+\gamma\big(\nabla \hat{u}_N,
	\nabla (\phi_l(\tilde{x})\phi_m(\tilde{y}))\big)
	-\big(\tilde{\psi}(\tilde{x},\tilde{y})\hat{u}_N,\phi_l(\tilde{x})\phi_m(\tilde{y})\big) \\
	&-\big(\hat{f}(\tilde{x},\tilde{y},t),\phi_l(\tilde{x})\phi_m(\tilde{y})\big)=0,\;\;l,m=0,1,\ldots,N-2.
	\end{aligned}
	\end{equation}
	Denote ${\bs \alpha}=(\alpha_{kj})_{k,j=0,1,\ldots,N-2},$ and
	\begin{equation*}\begin{split}
	\widetilde{\bs B}=(\tilde{b}_{kj})_{k,j=0,1,\ldots,N-2},\;\;\hat{f}_{kj}= \big(\hat{f}(\tilde{x},\tilde{y},t),\phi_k(\tilde{x})\phi_j(\tilde{y})\big),\;\;\widehat{\bs F}(t)=(\hat{f}_{kj})_{k,j=0,1,\ldots,N-2}.
	\end{split}\end{equation*}
	The \eqref{GLP-3} is equivalent to the following
	matrix equation
	\begin{equation*}
	-\ri\widetilde{\bs B}\bs \alpha'\widetilde{\bs B}+\gamma(\widetilde{\bs B}\bs \alpha \bs I_{N-1}+ \bs I_{N-1}
	\bs \alpha \widetilde{\bs B})-\bs W_{\bs \alpha}-\widehat{\bs F}(t)=0,
	\end{equation*}
	where ${\bs \alpha}=(\alpha_{kj})_{k,j=0,1,\ldots,N-2}$ and
	\begin{equation*}
	(\bs W_{\bs \alpha})_{lm} =  \sum_{k,j=0}^{N-2}
	\alpha_{kj}(t)\big(\tilde{\psi}(\tilde{x},\tilde{y})\phi_k(\tilde{x})\phi_j(\tilde{y}),\phi_l(\tilde{x})\phi_m(\tilde{y})\big),
	\  l,m=0,1,\ldots,N-2.
	\end{equation*}
	Using Lemma \ref{lemmasp1}, we have
	\begin{equation*}
	-\ri \widetilde{\bs B}\otimes \widetilde{\bs B}\, {\rm vec}(\bs \alpha')+(\gamma\widetilde{\bs B}
	\otimes  I_{N-1}+   \gamma \bs I_{N-1} \otimes
	\widetilde{\bs B}-\widetilde{\bs W}){\rm vec}(\bs \alpha)-{\rm vec} (\widehat{\bs F}(t))=0,
	\end{equation*}
	where
	$ \widetilde{\bs W}_{(l-1)(N-2) + m,(k-1)(N-2) + j} =
	\big(\tilde{ \psi}(\tilde{x},\tilde{y})\phi_k(\tilde{x})\phi_j(\tilde{y}),\phi_l(\tilde{x})\phi_m(\tilde{y})\big),\
	l,m,k,j=0,1,\ldots,N-2.
	$
	
	Setting
	\begin{equation*}
	\hat{u}_N(\tilde{x},\tilde{y},t_0)=\sum_{k,j=0}^{N-2}\alpha_{kj}(t_0)\phi_k(\tilde{x})\phi_j(\tilde{y}),
	\end{equation*}
	and with
	\begin{equation*}
	\big(\hat{u}_N(\tilde{x},\tilde{y},t_0)-\hat{\varphi}(\tilde{x},\tilde{y}),\phi_l(\tilde{x})\phi_m(\tilde{y})\big)=0,\;\;l,m=0,1,\ldots,N-2,
	\end{equation*}
	we obtian
	\begin{equation*}
	\widetilde{\bs B}\bs\alpha(t_0)\widetilde{\bs B}=\hat{\bs u}_0,\;\; (\hat{\bs u}_0)_{kj}=\big(\varphi(\tilde{x},\tilde{y}),\phi_k(\tilde{x})\phi_j(\tilde{y})\big).
	\end{equation*}
	Using Lemma \ref{lemmasp1}, we have
	\begin{equation*}
	\widetilde{\bs B}\otimes \widetilde{\bs B}\, {\rm vec}(\bs \alpha(t_0))= {\rm vec}(\hat{\bs u}_0).
	\end{equation*}
	Therefore, \eqref{GLP-1} leads to the following system of linear ordinary
	differential equations
	\begin{align}
	\label{GLP-4}
	\begin{cases}
	-\ri\widetilde{\bs B}\otimes \widetilde{\bs B} {\rm vec}(\bs \alpha')+(\gamma\widetilde{\bs B}\otimes  \bs I_{N-1}+ \gamma\bs  I_{N-1}\otimes\widetilde{\bs B}-\widetilde{\bs W}){\rm vec}(\bs \alpha)\\
	\hspace{5cm}-{\rm vec} (\widehat{\bs F}(t))=0,\quad  t_0<t\leq T, \\
	{\rm vec}(\bs \alpha(t_0))=\big(\widetilde{\bs B}\otimes \widetilde{\bs B}\big)^{-1}{\rm vec}(\hat{\bs u}_0).
	\end{cases}
	\end{align}
	Hence we can use the 3-stage IRK method for \eqref{GLP-4}.
	\section{A priori error estimate}

	In this section, we derive optimal a priori error bound for the
	semidiscrete scheme of the problem \eqref{Introduction-1} by using
	the Galerkin-Legendre spectral method discretization for the space. Recall the spaces
	\begin{equation*} \begin{split}
	L^2 (\widetilde{\Omega}) =\{u:(u,u) &<+\infty\},
	\;\;
	H^m (\widetilde{\Omega}) =\{u : {D}^k u \in L^2 (\widetilde{\Omega}),\ 0\leq |k|\leq m \}, \\
	H^m_{0 }(\widetilde{\Omega}) =&\{u \in H^m(\widetilde{\Omega}):
	{D}^k u (\partial \widetilde{\Omega} )=0,\ 0\leq |k|\leq m-1 \}.
	\end{split} \end{equation*}
	The norms in $L^2 (\widetilde{\Omega})$
	and $H^m (\widetilde{\Omega})$  denoted by $\|\cdot\| $ and $\|\cdot\|_{m}$,
	respectively, which are given as
	\begin{equation*}\|u\| = (u, u) ^{1/2},\quad \|u\|_{m}=
	\Big (\sum_{|k|=0}^m \|{D}^k u\| ^2\Big)^{1/2}.
	\end{equation*}
	Furthermore, we shall use
	$|u|_{m}= \Big(\sum\limits_{|k|=m}\|{D}^ku\| ^2\Big)^{1/2}$
	to denote the semi-norm in $H^m (\widetilde{\Omega})$.
	We now introduce the bochner space $L^p(J;X)$ endowed with the norm
	\begin{align*}
	\label{bochner}
	\|v\|_{L^p;X} =
	\begin{cases}
	(\int_J\|v\|_X^pdt)^{1/p},\,\,&1\le p<\infty,\\
	\textup{ess}\sup_{t\in J}\|v\|_X,\,\,&p=\infty, \\
	\end{cases}
	\end{align*}
	where $X=L^2(\widetilde{\Omega})$ or $X=H^m(\widetilde{\Omega})$.
	
	Let $H^{1}_{0}(\widetilde{\Omega})=H^1(\widetilde{\Omega})\bigcap
	\{v|v(\partial\widetilde{\Omega})=0\}.$ The $L^2(\widetilde{\Omega})$ orthogonal
	projection $\Pi^0_N:H^{1}_{0}(\widetilde{\Omega})\rightarrow
	V^2_N$ is defined by
	\begin{equation*}
	(\Pi^0_N  {u}- {u},v)=0, \quad\forall v \in V_N^2.
	\end{equation*}
	The operators $\Pi^0_N$ have the following approximation
	properties.
	\begin{lemma}\textup{(\cite[PP. 309 and 310]{Canuto1987Book})}
		\label{lemmaE1} For any positive integer $m\ge 1$, the following
		estimate holds for any $ {u}\in H^m(\widetilde{\Omega})\bigcap
		H^1_{0}(\widetilde{\Omega})$
		\begin{equation*}
		N\| {u}-\Pi^0_N
		{u}\|+\|\nabla( {u}-\Pi^0_N {u})\|\le C
		N^{1-m}\| {u}\|_{m},
		\end{equation*}
		where $C$ is independent of $N$.
	\end{lemma}
	
	We now split the error as a sum
	of two terms
	\begin{align}
	\label{EE1} \hat{u}_N-\hat{u} =(\hat{u}_N-\mathbb{P}^0_N \hat{u})+(\mathbb{P}^0_N \hat{u}-\hat{u})= \theta + \rho,
	\end{align}
	where $\mathbb{P}^0_N\hat u$ is an elliptic projection in $V_N^2$ of the exact
	solution $\hat{u}$, which  defined by
	\begin{align}
	\label{EE2} \mathcal{A}(\mathbb{P}^0_N \hat{u}-\hat{u}, v)=0,\quad \forall v \in V_N^2.
	\end{align}
	We begin with the following auxiliary result.
	\begin{lemma}\textup{(\cite[P. 309]{Canuto1987Book}}
		\label{lemmaE2} Assume that
		$\hat{u}\in  L^\infty(J, H^m(\widetilde{\Omega}) \cap H^1_{0}(\widetilde{\Omega}))$
		and $\mathbb{P}^0_N$ defined by \eqref{EE2}.
		Then
		\begin{equation*}
		N  \|\mathbb{P}^0_N\hat{u}-\hat{u}\|+   \|\nabla(\mathbb{P}^0_N\hat{u}-\hat{u})\|\leq CN^{1-m}\|\hat{u}\|_{m},\ \textup{for} \ \  t \in \bar{J},
		\end{equation*}
		where $C$ is independent of $N$.
	\end{lemma}
	
	\begin{lemma}
		\label{lemmaE3} With $\mathbb{P}^0_N$ defined by \eqref{EE2} and
		$\rho=\mathbb{P}^0_N\hat{u}-\hat{u}$. Assume that $\hat{u},\partial
		{_t}\hat{u}\in  L^\infty(J, H^m(\widetilde{\Omega}) \cap H^1_{0}(\widetilde{\Omega}))$. Then
		\begin{equation*}
		\|\rho(t)\|\le
		CN^{-m}\|\hat{u}\|_{m},\ \textup{for} \  t \in \bar{J},
		\end{equation*}
		\begin{equation*}
		\|\partial
		{_t}\rho(t)\|\le
		CN^{-m}\|\partial
		{_t}\hat{u}\|_{m},\ \textup{for} \ \ t \in \bar{J},
		\end{equation*}
		where $C$ is independent of $N$.
	\end{lemma}
	
	We are now ready for the $L^2$-error estimate for the
	semidiscrete problem.
	\begin{theorem}
		\label{lemmaE5} Let $\hat{u}$ and $\hat{u}_N$ be the solutions of \eqref{1-3}
		and \eqref{GLP-1}, respectively. Assume that $\tilde{\psi}$ is a real function, $\tilde{\psi}\in L^{\infty}(\widetilde{\Omega})$ and $\hat{u},\ \partial
		{_t}\hat{u} \in  L^\infty(J, H^m(\widetilde{\Omega}) \cap H^1_{0}(\widetilde{\Omega}))$. Then, we have
		\begin{equation*}\begin{split}
		&\|\hat{u}_N(t)-\hat{u}(t)\| \leq CN^{-m}
		\big(\|\hat{u}\|_{L^\infty(J;H^m(\widetilde{\Omega}))}+\|\partial
		{_t}\hat{u}\|_{L^\infty(J;H^m(\widetilde{\Omega}))}\big),\ \textup{for}\ t \in J,\\
		\end{split}\end{equation*}
		where $C$ is independent of $N$.
	\end{theorem}
	\begin{proof}
		With $\hat{u}_N$ and $\hat{u}$ satisfies the following equation
		\begin{equation*}
		-\ri(\partial
		{_t}\hat{u}_N, v)+\gamma\mathcal{A}(
		\hat{u}_N, v)-(\tilde{\psi}(\tilde{x},\tilde{y})\hat{u}_N,v)-(\hat{f}(\tilde{x},\tilde{y},t),v)=0,\quad \forall v \in V_N^2,
		\end{equation*}
		and
		\begin{equation*}
		-\ri(\partial
		{_t}\hat{u}, v)+\gamma\mathcal{A}(
		\hat{u}, v)-(\tilde{\psi}(\tilde{x},\tilde{y})\hat{u},v)-(\hat{f}(\tilde{x},\tilde{y},t),v)=0,\quad \forall v \in V_N^2.
		\end{equation*}
		Hence with $v=\theta$, we obtain
		\begin{equation*}\begin{split}
		-\ri(\partial
		{_t}(\theta+\rho), v)+\gamma\mathcal{A}(
		\theta+\rho, v)-(\tilde{\psi}(\tilde{x},\tilde{y})(\theta+\rho),v)=0,\quad \forall v \in V_N^2.
		\end{split}\end{equation*}
		Taking both sides the imaginary, we obtain
		\begin{equation*}\begin{split}
		\frac{1}{2}\frac{d}{dt}\|\theta\|^2 = -{\rm Re}(\partial
		{_t}\rho, \theta)-{\rm Im}(\tilde{\psi}(\tilde{x},\tilde{y})\rho,\theta),
		\end{split}\end{equation*}
		such that
		\begin{equation*}
		\frac{1}{2}\frac{d}{dt}\|\theta\|^2\le C(\|\theta\|^2+\|\rho\|^2+\|\partial
		{_t}\rho\|^2).
		\end{equation*}
		After integration, this shows
		\begin{equation*}
		\|\theta(t)\|^2\le C\|\theta(t_{0})\|^2
		+C\int^t_{t_0}(\|\theta\|^2+\|\rho\|^2+\|\partial
		{_s}\rho\|^2)ds.
		\end{equation*}
		By using the Gronwall's Lemma, we have
		\begin{equation*}
		\|\theta(t)\|^2\le C\|\theta(t_0)\|^2
		+C\int^t_{t_0}(\|\rho\|^2+\|\partial
		{_s}\rho\|^2)ds.
		\end{equation*}
		By Lemma \ref{lemmaE3}, we have
		\begin{equation*}\begin{split}
		\|\theta(t)\|^2\le &C\|\theta(t_0)\|^2 +CN^{-2m}(T-t_0) \big(\|\hat{u}\|_{L^\infty(J;H^m(\widetilde{\Omega}))}^2+
		\|\partial
		{_t}\hat{u}\|_{L^\infty(J;H^m(\widetilde{\Omega}))}^2 \big)\\
		\le &C_5N^{-2m}\big(\|\hat{u}\|_{L^\infty(J;H^m(\widetilde{\Omega}))}^2
		+\|\partial
		{_t}\hat{u}\|_{L^\infty(J;H^m(\widetilde{\Omega}))}^2\big).
		\end{split}\end{equation*}
		Using  Lemmas \ref{lemmaE1} and \ref{lemmaE2}, we obtain
		\begin{equation*}\begin{split}
		\|\theta(t_0)\|=&\| \hat{u}_N(t_0)-\mathbb{P}^0_N \hat{u}(t_0)\|\le\|\Pi^0_N\hat{u}(t_0)-\hat{u}(t_0)\|+\|\mathbb{P}^0_N \hat{u}(t_0)-\hat{u}(t_0)\|
		\\
		\le& CN^{-m}\big(\|\hat{u}\|_{L^\infty(J;H^m(\widetilde{\Omega}))}
		+\|\partial
		{_t}\hat{u}\|_{L^\infty(J;H^m(\widetilde{\Omega}))}\big).
		\end{split}\end{equation*}
		Then we have
		\begin{equation*}\begin{split}
		\|\hat{u}_N(t)-\hat{u}(t)\|\le& \|\theta(t)\|+\|\rho(t)\|\le
		CN^{-m}\big(\|\hat{u}\|_{L^\infty(J;H^m(\widetilde{\Omega}))}
		+\|\partial
		{_t}\hat{u}\|_{L^\infty(J;H^m(\widetilde{\Omega}))}\big),
		\end{split}\end{equation*}
		this complete the proof.
	\end{proof}
	
	\section{Numerical results}
	
	In this section, we present numerical examples to demonstrate the convergence and accuracy of the new method.
	
	For a given $N$, we denote the discrete $L^2$-error by
	\begin{equation*}
	\textup{error}(t) =\Big( \sum_{k=0}^{N}\sum_{k=0}^{N}\big(v_N(x_k,y_m,t)-v(x_k,y_m,t)\big)^2\omega_{km}\Big)^{\frac{1}{2}}.
	\end{equation*}
	where $v =\text{Re}(u)$ or $v =\text{Im}(u)$, $x_k$, $y_m$ are Legendre-Gauss-Lobatto quadrature nodes,
	$\omega_{km}$ are quadrature weights in $\Omega$ (see, e.g., \cite[Theorem 3.29]{Shen2011Book})
	
	\subsection{Test Problem 1}
	We consider problem \eqref{Introduction-1} with $c= 0$, $d= 1$, $t_0=0$, $T=1$, $\psi(x,y) =3-2\tanh^2(x)-2\tanh^2(y)$ and the following initial condition
	\begin{equation*} \phi(x,y) = \frac{\ri}{\cosh(x)\cosh(y)},\ x \in (c,d)\times (c,d).
	\end{equation*}
	The exact solution is given by
	\begin{align}
	\label{test-1}
	u(x,y,t) =\frac{\ri\exp(\ri t)}{\cosh(x)\cosh(y)}.
	\end{align}
	The boundary conditions can be obtained easily from \eqref{test-1}.
	
	This test problem is given in \cite{Mohebbi2009paper}.
	Figure \ref{fig-1} shows the
	surface plot of absolute error for Test Problem 1 with $N=18$ and $h=1/20$ at $T=1$.
	Figure \ref{fig-2} shows the
	surface plot of absolute error for Test Problem 1
	by using the method of \cite{Mohebbi2009paper} with $N_{x}=N_{y}=\frac{d-c}{ \Delta x } =\frac{d-c}{ \Delta y } =20$ and $\Delta t=1/20$.
	Comparing Figure \ref{fig-1} with Figure \ref{fig-2}, we can see that our method
	is more accurate than the algorithm of \cite{Mohebbi2009paper}.
	Figure \ref{fig-3} shows the convergence rates for Test problem 1 with $h = 1/50$ at $T=1$.
	The curves show exponential rates of convergence in space.
	\begin{figure}[!htp]
		\centering
		\includegraphics[width=7cm]{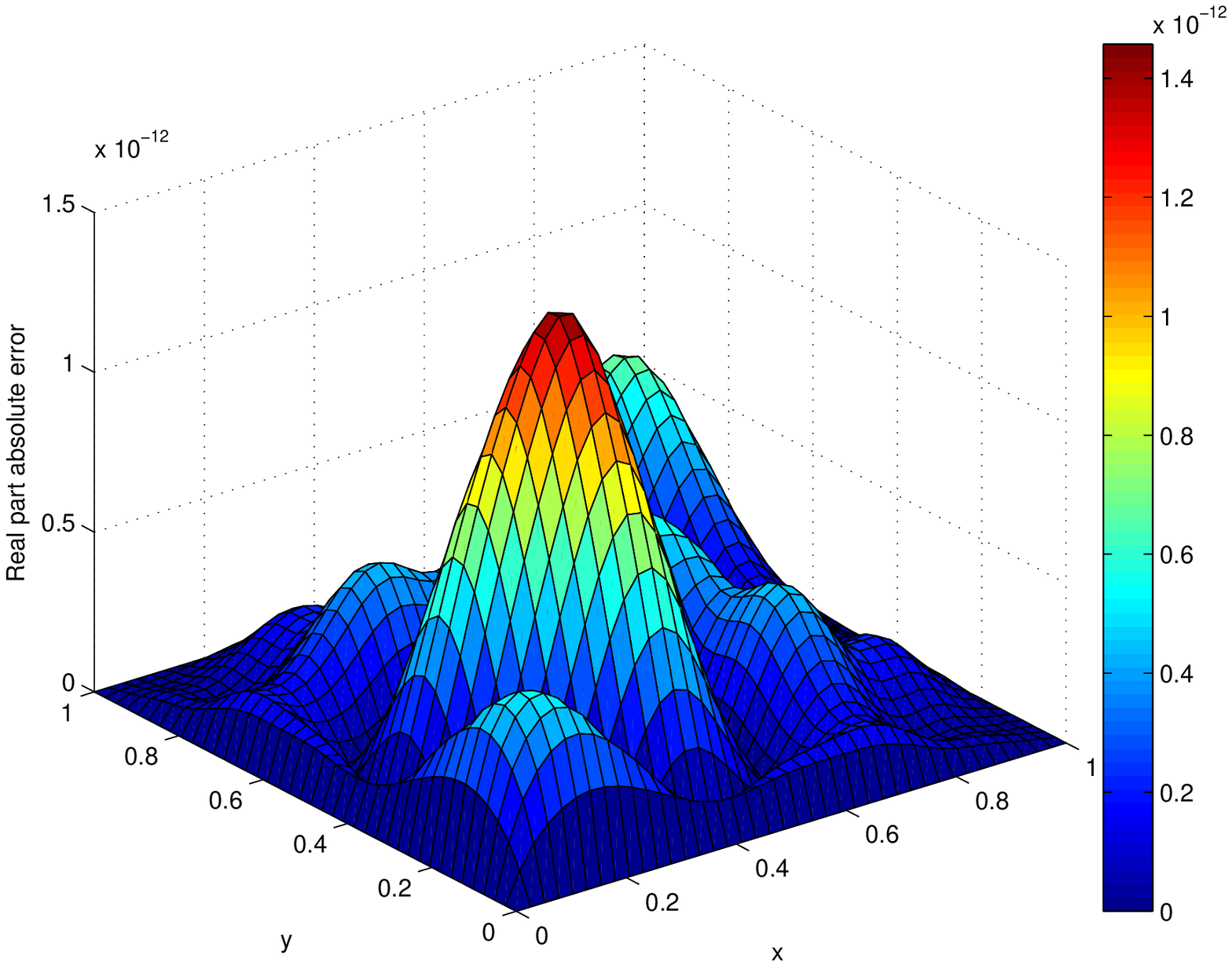}
		\quad
		\includegraphics[width=7cm]{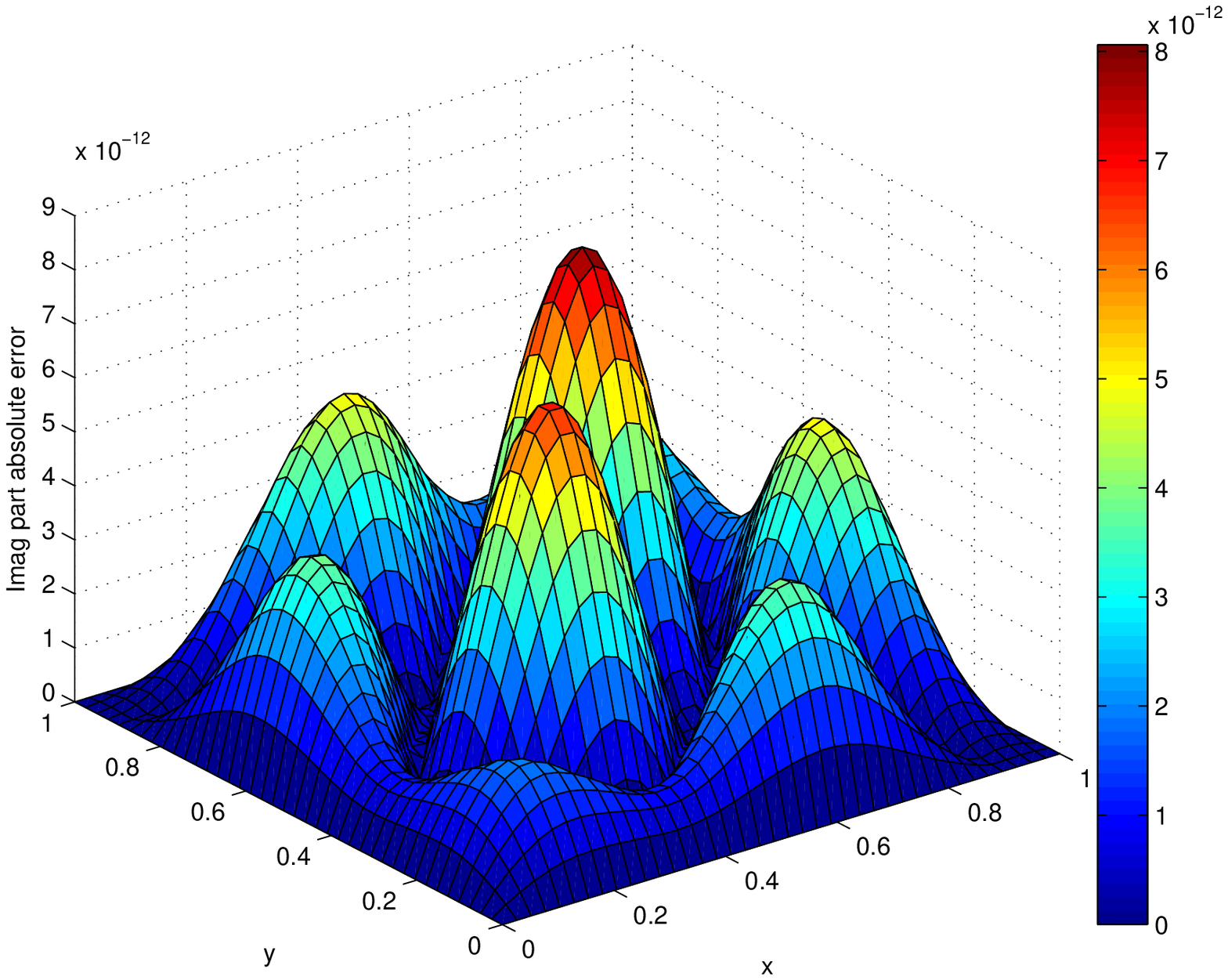}
		\caption{Surface plot of absolute error obtained for Test problem 1 in time with $N=18$ and $h=1/20$ (left panel for real part and right panel for imaginary part).}
		\label{fig-1} 
	\end{figure}
	\begin{figure}[!htp]
		\centering
		\includegraphics[width=7cm]{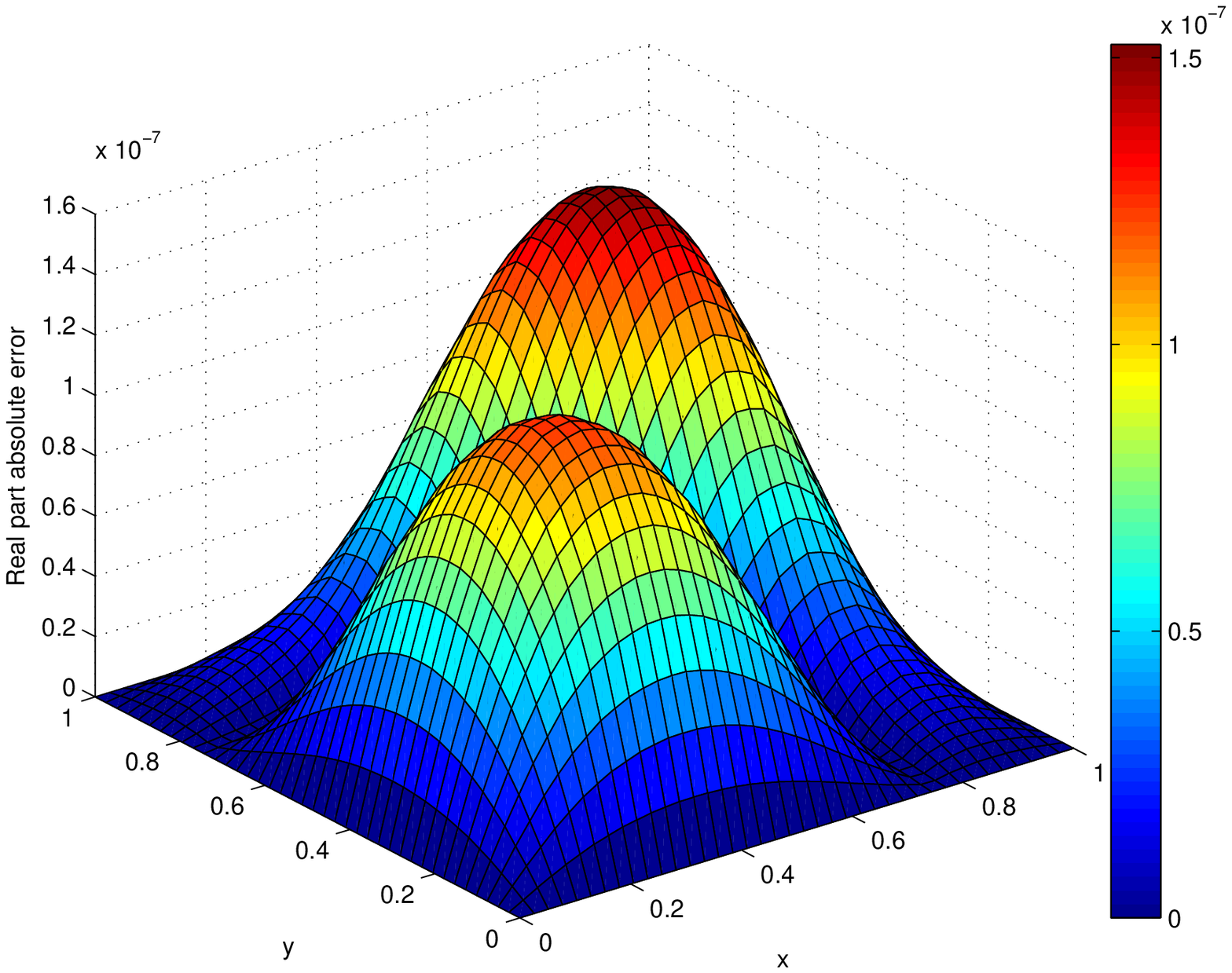}
		\quad
		\includegraphics[width=7cm]{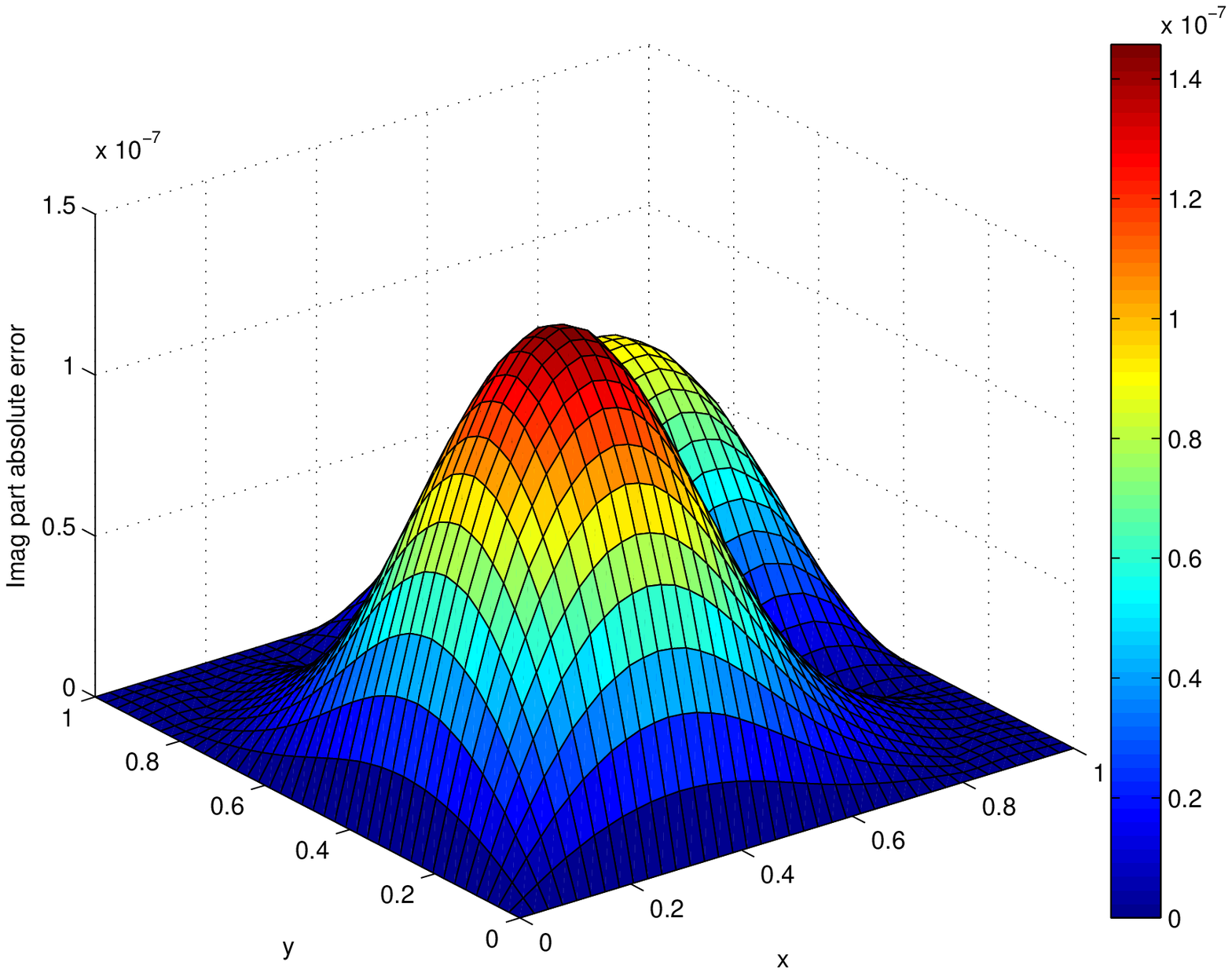}
		\caption{Surface plot of absolute error obtained for Test problem 1 by using the method of \cite{Mohebbi2009paper} with $N_x=N_y=20$ and $\triangle t=1/20$ (left panel for real part and right panel for imaginary part).}
		\label{fig-2} 
	\end{figure}
	\begin{figure}[!htp]
		\centering	
		\includegraphics[width=7cm]{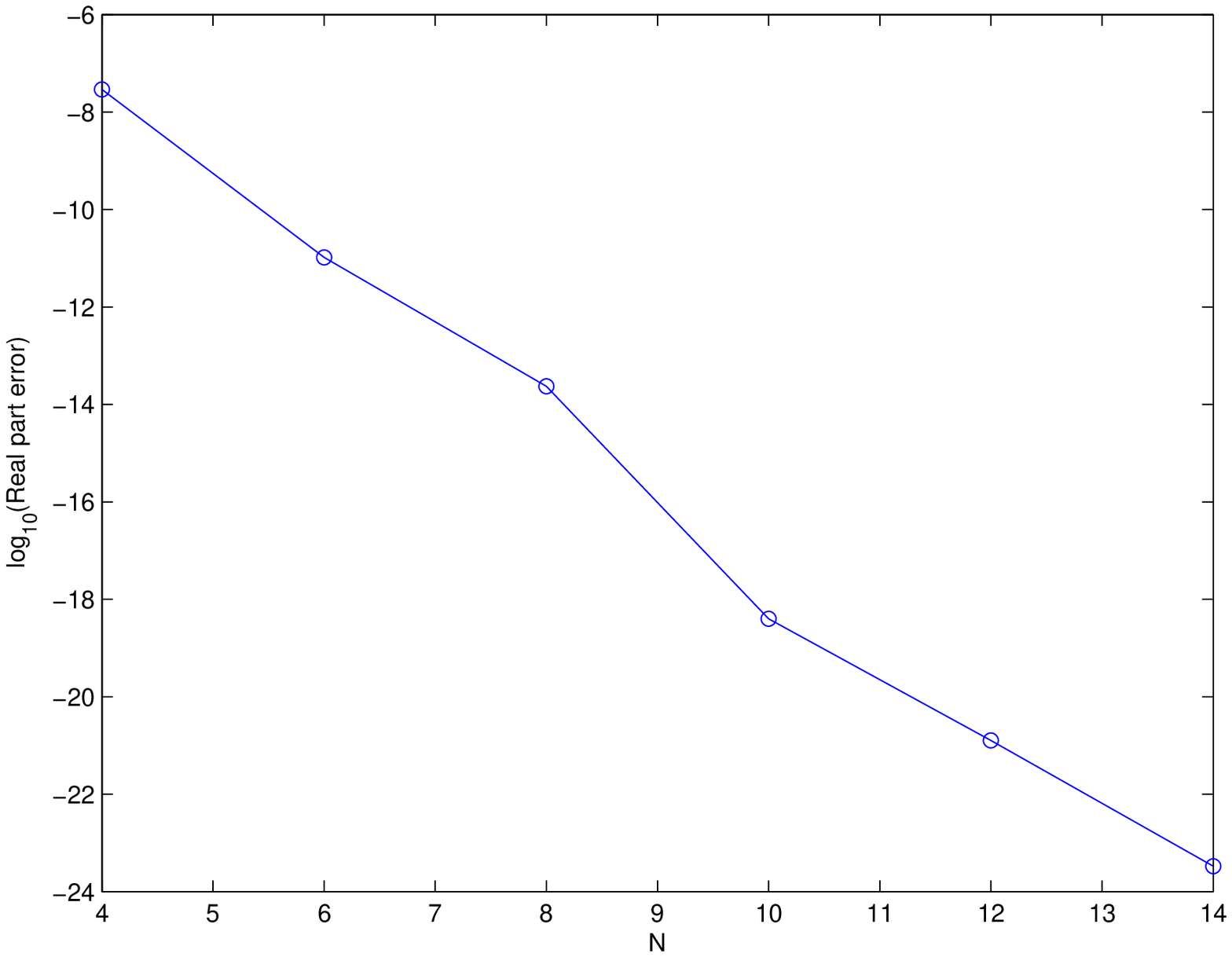}
		\quad
		\includegraphics[width=7cm]{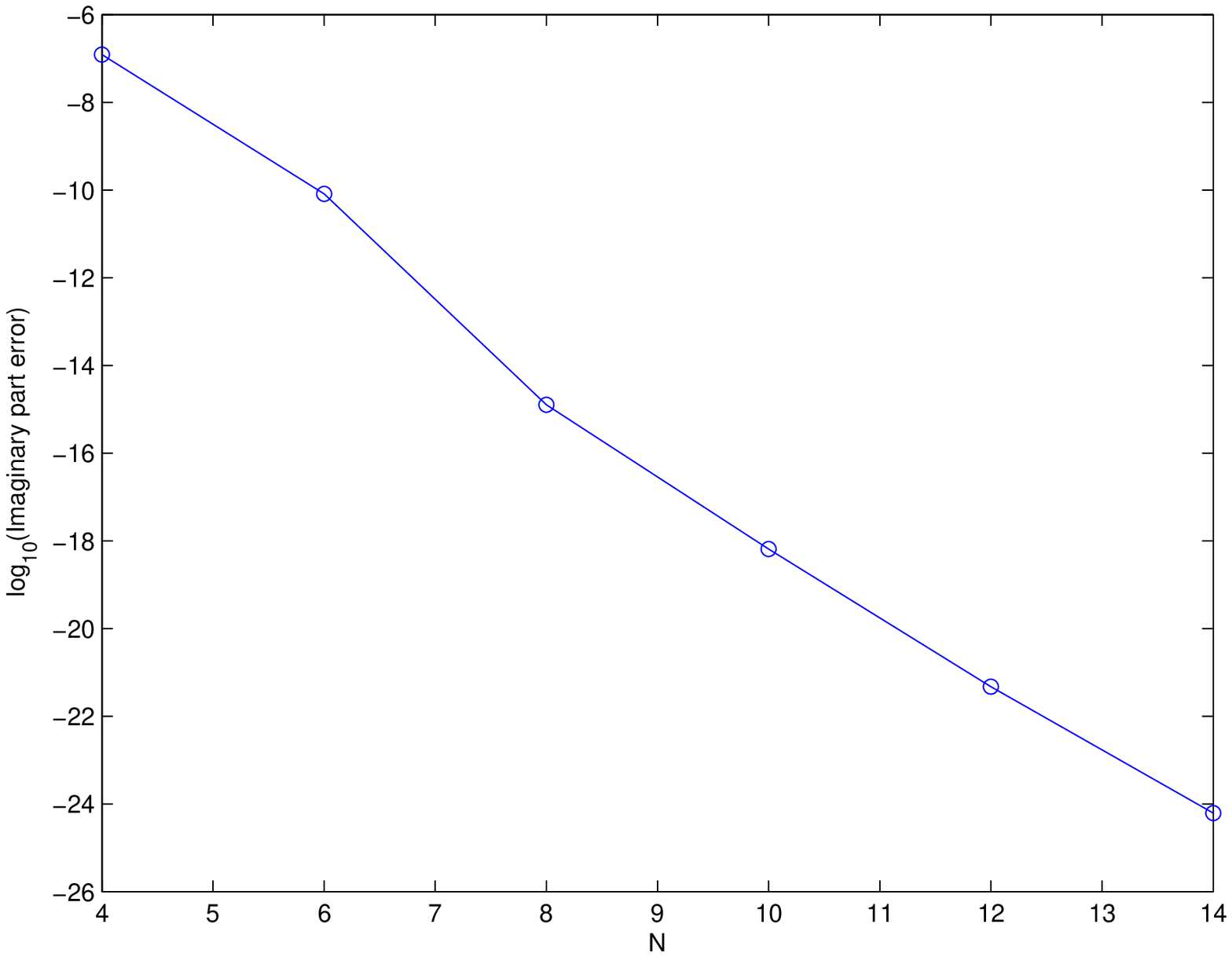}
		\caption{Exponential convergence in $N$ for Test problem 1 with $h =1/50$ (left panel for real part and right panel for imaginary part).}
		\label{fig-3} 
	\end{figure}
	
	\subsection{Test Problem 2}
	We consider problem \eqref{Introduction-1} with $c= -2.5$, $d= 2.5$, $t_0=0$, $T=1$, $\psi(x,y) =0$ and the following initial condition
	\begin{equation*} \phi(x,y) = \exp\big(-\ri k_0x-(x^2+y^2)\big),\ x \in (c,d)\times (c,d).
	\end{equation*}
	The exact solution is given by
	\begin{align}
	\label{test-2}
	u(x,y,t) =\frac{\ri}{\ri-4t} \exp\Big(   -\frac{\ri(x^2+y^2+\ri k_0x+\ri k_0^2 t)}{\ri-4t}\Big ).
	\end{align}
	The boundary conditions can be obtained easily from \eqref{test-2}.
	
	This test problem is given in \cite{Mohebbi2009paper}.
	Figure \ref{fig-4} shows the
	surface plot of absolute error for Test Problem 2
	with $N=25$ and $h=1/25$ at $T=1$.
	Figure  \ref{fig-5} shows the convergence rates for Test problem 2 with $h =1/100$ at $T=1$.
	Table \ref{tab-1} shows the
	absolute errors for Test Problem 2
	with $N=25$ and $h=1/20$.
	Table \ref{tab-2} shows the
	absolute errors for Test Problem 2
	by using the method of \cite{Mohebbi2009paper} with
	$N_{x}=N_{y}=\frac{d-c}{\Delta  x } =\frac{d-c}{\Delta  y } =50$ and $\Delta  t=1/20$ (see \cite[Tab. 8]{Mohebbi2009paper}).
	Comparing Table \ref{tab-1} with Table \ref{tab-2}, we can see our method
	is better than the method of the \cite{Mohebbi2009paper}.
	\begin{figure}[!htp]
		\centering
		\subfigure{
			\includegraphics[width=7cm]{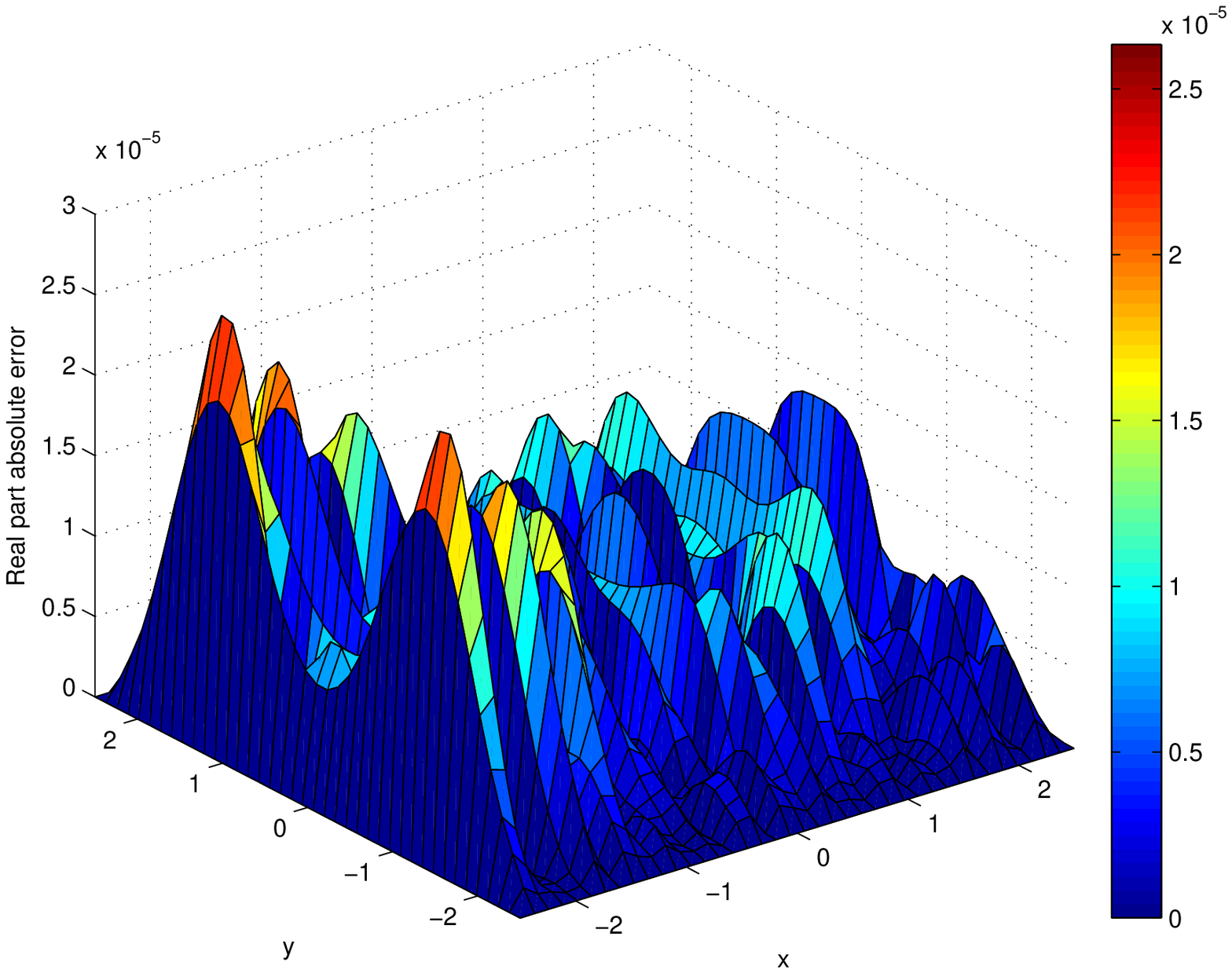}}
		\subfigure{
			\includegraphics[width=7cm]{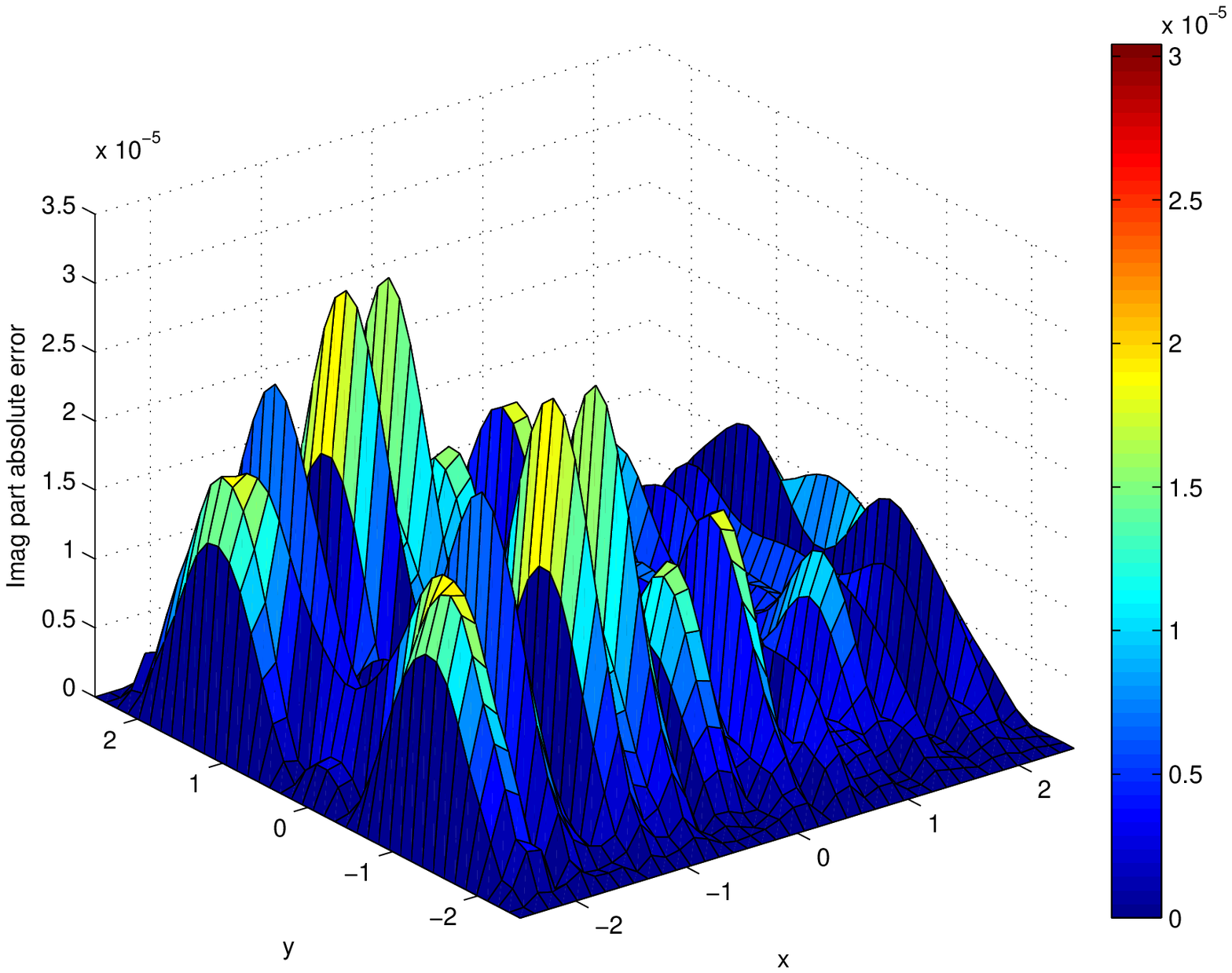}}
		\caption{Surface plot of absolute error obtained for Test problem 2 with $N=25$ and $h=1/25$ (left panel for real part and right panel for imaginary part).}
		\label{fig-4} 
	\end{figure}
	\begin{figure}[!htp]
		\centering
		\includegraphics[width=7cm]{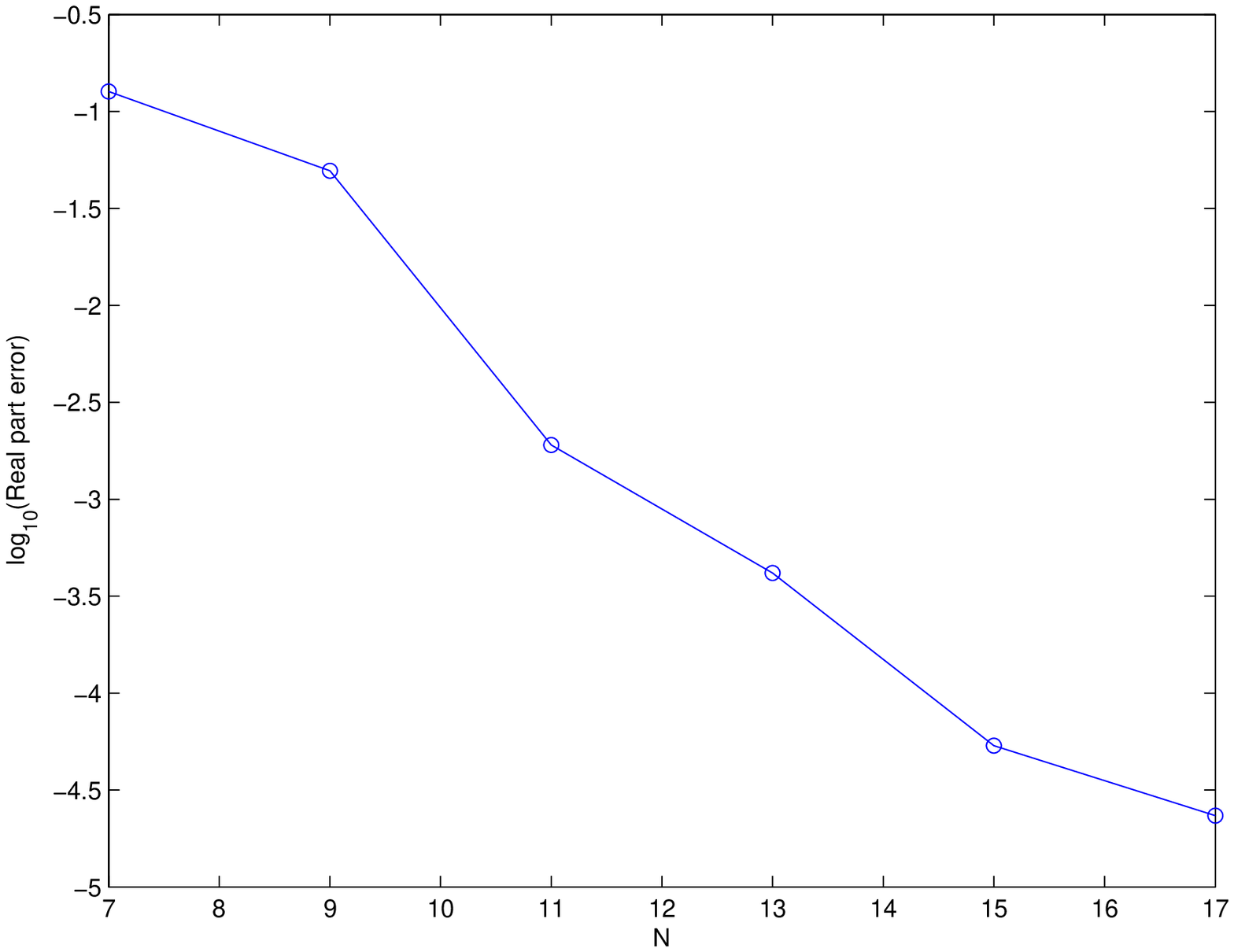}
		\quad
		\includegraphics[width=7cm]{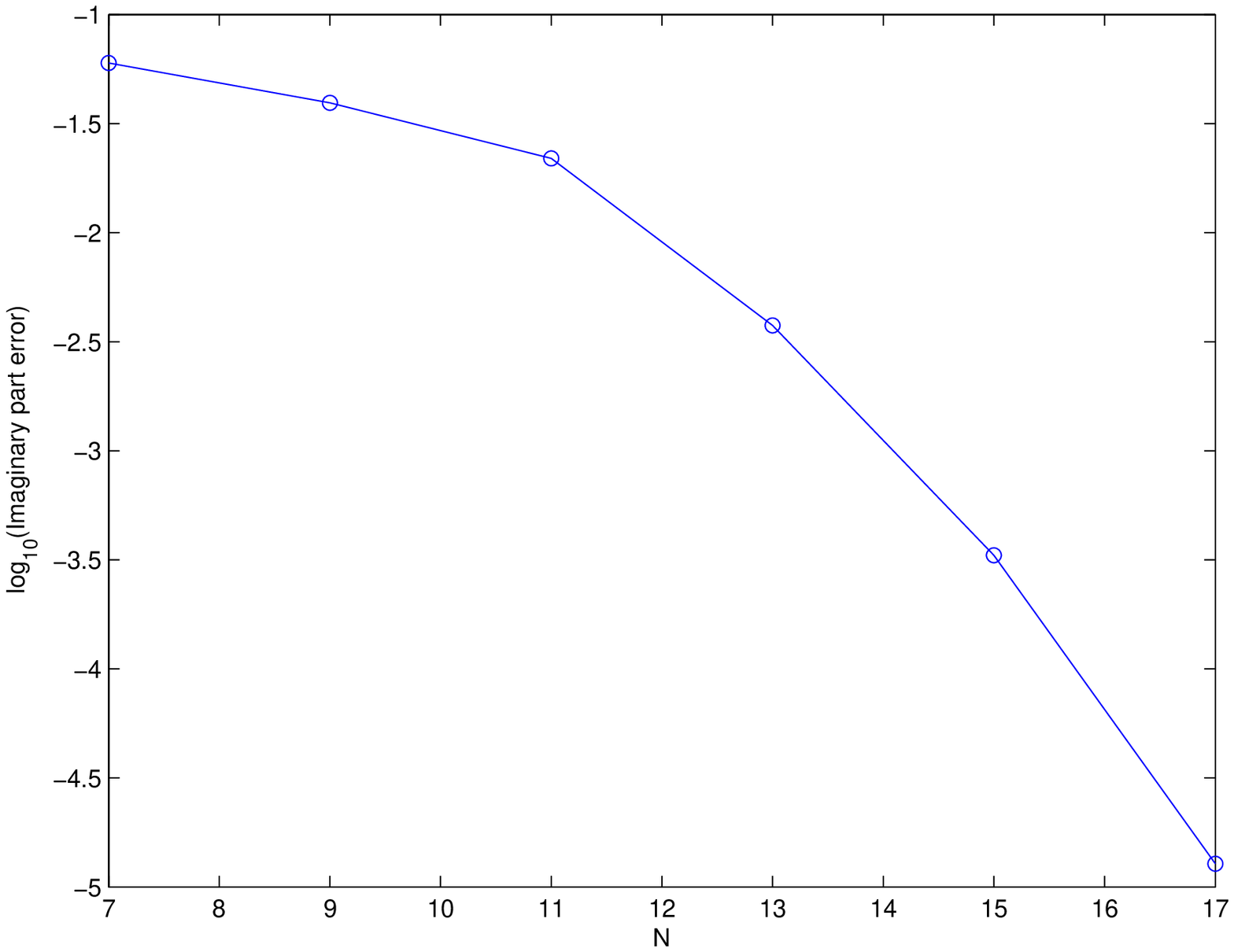}
		\caption{Exponential convergence in $N$ for Test problem 2 with $h =1/100$ (left panel for real part and right panel for imaginary part).}
		\label{fig-5}
	\end{figure}
\begin{table}[!htp]
	\footnotesize
	\caption{Absolute errors for solving Test problem 2 with $N=25$ and $h =1/20$.}
	\label{tab-1}
	\centering
	\begin{tabular}
		{ccccccccccc} \hline
		$t$ &\multicolumn{2}{c}{Maximum absolute error} &\multicolumn{2}{c}{Average absolute error}
		\\
		\cmidrule(lr){2-3}\cmidrule(lr){4-5}
		\, &Real part &Imaginary part &Real part &Imaginary part
		\\ \hline
		0.10 & 5.5837e-05 & 7.2420e-05 & 5.2562e-06 & 6.5224e-06 \\
		0.25 & 1.1025e-04 & 1.6687e-04 & 1.3543e-05 & 1.2252e-05 \\
		0.50 & 6.4010e-05 & 6.5695e-05 & 1.8633e-05 & 1.8118e-05 \\
		0.75 & 6.6335e-05 & 8.7873e-05 & 1.8833e-05 & 1.8476e-05 \\
		1.00 & 8.9998e-05 & 9.2257e-05 & 1.4600e-05 & 1.6865e-05 \\
		\hline
	\end{tabular}
\end{table}

\begin{table}[!htp]	
	\footnotesize
	\caption{Absolute error for solving Test problem 2 by using the method of \cite[Table 8]{Mohebbi2009paper} with $N_x=N_y=50$ and $\Delta t=1/20$.}
	\label{tab-2}
	\centering
	\begin{tabular}
		{ccccccccccc} \hline
		t &\multicolumn{2}{c}{Maximum absolute error} &\multicolumn{2}{c}{Average absolute error}
		\\
		\cmidrule(lr){2-3}\cmidrule(lr){4-5}
		\, &Real part &Imaginary part &Real part &Imaginary part
		\\ \hline
		0.10 & 5.6517e-05 & 5.7493e-05 & 9.4025e-06 & 9.0740e-06 \\
		0.25 & 2.6182e-04 & 1.1500e-04 & 2.2048e-05 & 1.1013e-05 \\
		0.50 & 1.2792e-04 & 1.3972e-04 & 2.8362e-05 & 3.4913e-05 \\
		0.75 & 1.3312e-04 & 1.2511e-04 & 3.5252e-05 & 2.8212e-05 \\
		1.00 & 1.3647e-04 & 9.8227e-05 & 3.7975e-05 & 3.1640e-05 \\
		\hline
	\end{tabular}
\end{table}

	\section{Concluding remarks}
	In the paper, we proposed a Galerkin-Legendre spectral method with implicit Runge-Kutta method
	for two-dimensional linear Schr\"{o}dinger equation
	with the nonhomogeneous Dirichlet boundary conditions and initial condition.
	Optimal a priori error bounds are derived in the
	$L^2$-norm for the semidiscrete formulation.
	Our numerical results confirm the exponential convergence
	in space.
	
	
\end{document}